\documentclass[11pt]{amsart}

\usepackage{ latexsym}

\usepackage{verbatim, amscd,
epsfig,amsmath,amsthm,amstext,amssymb,hyperref
}
\usepackage[all]{xy}
\title {Spectral Curves for Almost-Complex Tori in $ S ^ 6 $}
\author{Emma Carberry, Erxiao Wang}
\address{ Emma Carberry\\Department of Mathematics\\University of Sydney
 \\ NSW, 2006, Australia} \email{carberry@maths.usyd.edu.au}
\address {Erxiao  Wang \\Department of Mathematics\\National University of Singapore\\2, Science Drive 2\\
Singapore 117543}\email {matwe@nus.edu.sg}
\date{\today}

\theoremstyle{plain}
\newtheorem{theorem}{Theorem}[section]
\newtheorem{lemma}[theorem] {Lemma}
\newtheorem{corollary}[theorem] {Corollary}
\newtheorem{proposition}[theorem]{Proposition}

\newtheorem{remark} [theorem] {Remark}
\theoremstyle{definition}
\newtheorem{definition}[theorem]{Definition}

\newcommand{\R}{\mathbb{ R}}
\newcommand{\C}{\mathbb{ C}}
\newcommand{\Z}{\mathbb{ Z}}

\renewcommand{\P}{\mathbb{P}}
\renewcommand{\O}{\mathbb{O}}

\renewcommand{\P}{\mathbb{ P}}

\newcommand{\CP}{\C\P}

\newcommand {\holomorphic} { \mathcal {O}}
\newcommand {\degree} {(6k+1)}

\newcommand{\g}{\mathfrak g}
\renewcommand{\loop}{\Lambda^{\rho,\tau}\g_2 ^\C}

\numberwithin{equation}{section}

\theoremstyle{plain}

\begin {comment}
\newtheorem{thm}{Theorem}[section]

\theoremstyle{definition}

\newtheorem{remark}[thm]{Remark}

\end {comment}

\newcommand{\Pj}{\mathbb{P}}

\newcommand{\sx}{ {\scriptscriptstyle X} }

\newcommand{\la}{\lambda}

\newcommand{\ce}{\mathcal{E}}

\newcommand{\co}{\mathcal{O}}

\DeclareMathOperator{\End}{\mathrm{End}}

\begin{document}
\maketitle

To each non-isotropic almost-complex immersion of a 2-torus into $ S
^ 6 $ we associate an algebraic curve, called the spectral curve,
and a linear flow in the intersection of two Prym varieties on this
spectral curve.   We show that generically the spectral curve is
smooth and compute the dimension of the moduli space of such curves
and of the torus in which the eigenline bundles lie.

\section {introduction}

The six-sphere has a natural almost-complex structure \mbox {$J_x:
T_xS ^6\rightarrow T_xS ^ 6 $} given by \mbox {$J_x (y) = x\cdot y
=x\times y $}, where we consider $ S^6$ as a subspace of the
imaginary octonions and $\cdot $ denotes octonionic multiplication.
An immersion $f: M ^ 2\to S ^ 6$ of a Riemann surface $ M ^ 2 $ is
{\it almost-complex} if $df\circ J _M =J\circ df$, where $ J _ M $
denotes the complex structure on $M$.  This can be interpreted in
terms of a calibrated geometry:
  $f $ is almost-complex precisely when the cone $C$ over $f (M^2)$ is an {\it associative} (singular) submanifold of Im~$\O $.  This means that $ C $ is calibrated with respect to the 3-form
\[
\alpha '(x, y, z) = <x\cdot y, z>
\]
(and hence volume minimising in its homology class) or equivalently
that  $1\oplus T_p(C)$ is an associative subalgebra of $\O $.

Let $f: M ^2\rightarrow S ^ 6 $ be an almost-complex immersion. Such
an $ f $ is necessarily minimal, and the language of harmonic
sequences can be used to classify the types of almost-complex curves
in $ S^6$ \cite {BVW:94}.  The {\it isotropic} (or {\it
superminimal}) immersions can be described in terms of holomorphic
data.  Either they are  totally geodesic and thus given by the
intersection of $ S ^ 6 $ with an associative 3-plane or have been
shown by Bryant to possess a Weierstrass type representation \cite
{Bryant:82:O}.  We show that when $ M ^ 2 $ is a torus, the
non-isotropic immersions also give holomorphic data, by means of a
spectral curve construction.  These immersions are all
superconformal, and either linearly full in $ S ^ 6 $ or linearly
full in a totally geodesic $ S ^ 5\subset S ^ 6 $. In the latter
case, the almost-complex condition says precisely that $ f $ is a
minimal Legendrian immersion. In the case of tori, various authors
\cite {Sharipov:91, MaMa:01, McIntosh:02} have described these
immersions in terms of  algebro-geometric ({\it spectral curve})
data and the periodicity problem has also been solved \cite {CM:03}.
This showed that  minimal Legendrian tori are abundant in the sense
that for every positive integer $n $, there are countably many real
$n $-dimensional families of them. of every real dimension.

We shall describe the spectral data for superconformal
almost-complex tori $ f: M ^ 2\rightarrow S ^ 6 $.  The  condition
that $ f $ is almost-complex is equivalent to the existence of a
primitive lift into the 6-symmetric space $ G _ 2/T ^ 2 $, where $ T
^ 2 $ denotes the maximal torus of $ G _ 2 $. A superconformal
immersion which is merely minimal possesses a primitive lift into
$SO(7)/T^3$, for $T^3$ the maximal torus of $ SO (7)$.  The main
task is to characterise the subspace $G_2/T^2\subset SO(7)/T^3$ in
terms of symmetries of the spectral data. The group $G_2 $ is the
connected component of the subgroup of $ GL (7,\R) $ preserving a
generic 3--form  $\alpha ' $ \cite {Bryant:87}, for example the
calibration form given above.  In his study of Langlands duality for
$G_2 $ Higgs bundles \cite {Hitchin:06}, Hitchin used this to give a
beautiful description of $G_2$ Higgs bundles in terms of spectral
curve data.  He found that the  abelian variety for a $G_2$ Higgs
bundle is not a Prym variety but rather is given by the intersection
of two Prym varieties, the symmetries of which together characterise
$G_2$.  We show that a similar picture persists for superconformal
almost-complex tori $ f: M ^ 2\rightarrow S ^ 6 $; the eigenline
bundles give a linear flow in the intersection of two Prymians.  The
eigenline bundles also satisfy a reality condition since we
complexify the group $G_2$ in order to obtain the spectral data, and
there is also a natural sixfold symmetry by which we quotient.

\section{integrable systems formulation}
We begin by describing how a superconformal almost-complex immersion
$ f: M^2 \rightarrow  S^6 $ yields a primitive lift into the
6-symmetric space $ G _ 2/T ^ 2 $. A non-isotropic branched minimal
immersion $\R^2\rightarrow S ^n $ of finite type is in fact an
immersion (this for example follows from the similar result proven
in \cite{McIntosh:95} for maps into complex projective space) and so
it is no restriction to ask $f $ to be an immersion. Recall that
$G_2\subset SO (7) $ is the automorphism group of the octonion
algebra, given explicitly by those matrices whose (orthonormal)
columns $f_1,\ldots, f_7 $ satisfy
\[
f_3 = f_1\cdot f_2,\, f_5 = f_1\cdot f_4,\, f_6= f_2\cdot f_4,\,
f_7= f_3\cdot f_4.
\]
The maximal torus $ T ^2$ of $ G_2 $ consists of matrices of the
form
\[
\mbox {diagonal} (1, R_{\theta}, R_{\phi}, R_{\theta+\phi})
\]
where the $ R_\theta$ denotes the standard rotation matrix with
angle $\theta $. Taking $ f_1 =f$ and $ f_2$ to be a unit tangent
vector then for any unit-length choice of $f_4\perp {\rm span}_\R
\{f, f_2, f_3\}$ the above equations define a $G_2 $ framing of the
almost-complex immersion $ f $.  We shall make use of a particular
such framing.

One may define an order six involution of $ G _ 2 $ by $\tau =\rm
{Ad} _C$, where
  \[
C: = (1, \, R_\frac {\pi} {3}, \, R_\frac {2\pi } {3},\, R_\pi).
\]
\begin {comment}
 \[C=\left (\begin{array} {ccccccc}
1 & & & & & &\\
&\cos{\frac{\pi} {3}} & -\sin{\frac\pi 3}  & & & &\\
&\sin\frac\pi 3 &\cos\frac\pi 3 & & & &\\
& & &\cos{\frac{2\pi} {3}} & -\sin{\frac {2\pi} { 3} } & &\\
& & &\sin\frac {2\pi} {3} &\cos\frac {2\pi} { 3}  & &\\
& & & & & -1 & 0\\
&& & & & 0 & -1\end{array}\right).\]

\end {comment}
Writing $\g ^\C: =\g\otimes \C $, $\tau $ gives the decomposition
\[
\g ^\C =\oplus _ {j= 0} ^5 \g _ j
\]
where $\g _ j $ is the $\exp {2\pi\sqrt {- 1}j/6} $-eigenspace of
$\tau $. The maximal torus $ T = T ^ 2 $ of $ G_2 $ is preserved by
$\tau $, and hence this also yields the decomposition
\[
T (G_2/T) ^\C=\oplus _ {j= 0} ^5 [\g _ j ].
\]

Recall that the harmonic sequence of a harmonic map from a surface
into a complex projective space is a sequence of line bundles (or
equivalently of maps into $\CP^n $) obtained by taking successive
derivatives
\[
f_{j+1}=\partial_j(\frac {\partial} {\partial z}\otimes f _ j),
\]
and that if some $k $ consecutive elements of this sequence are
mutually orthogonal then so are all sets of $k $ consecutive
elements. The harmonic sequence of $f:M^2\rightarrow S ^6 $ is
defined to be that of its natural lift into $\CP ^6 $ and each of
the maps in the sequence is harmonic.  This sequence enables one to
give a useful classification of harmonic maps from a surface into a
sphere or complex projective space.  The map $f $ is said to be {\it
isotropic} if its harmonic sequence is the Fr\^enet frame of a
linearly full holomorphic curve; this is necessarily the case if
seven consecutive elements of the sequence are orthogonal
\cite{BW:92}
 and such maps are by definition described by holomorphic data. We
  call $f:M^2\rightarrow S ^6 $ {\it superconformal} if precisely
  six consecutive members of its harmonic sequence are orthogonal,
   and every almost-complex $f:M^2\rightarrow S ^6 $ is either isotropic or superconformal \cite{BVW:94}. For further information on harmonic sequences in this context we refer the reader to \cite {BW:92, BPW:95, Burstall:95}.

Given a superconformal $f:M^2\rightarrow S ^6 $ we may then use
harmonic sequences to define a natural lift
\[ F= (f, f_2 , f_3, f_4, f_5, f_6, f_7)\]
so that
\[
[ F ]:\widetilde M ^ 2\rightarrow G_2/T
\] is well-defined.  In fact  we may write
\[ F ^ {-1} d F = (u_0 +u_1) dz + (\bar u_0+\bar u_1) d\bar z,\, u_j\in\g_j.\] This says that the lift $ [ F ] $ is primitive with respect to $\tau $; it satisfies this condition precisely when $ f $ is a superconformal almost-complex immersion \cite {BPW:95}.  Here $\widetilde M ^ 2 $ denotes the universal cover of $ M ^ 2 $.

$G_2 ^\C $ additionally possesses the antiholomorphic  involution
$\rho (g) =\bar g $, which commutes with $\tau  $. The form
$\varphi= F ^ {- 1} d F $ satisfies the Maurer-Cartan equation
\[ d\varphi +\frac 12 [\varphi,\varphi ]= 0\]
and one easily checks that
\[\varphi_\zeta:= (u_0+u_1\zeta) dz+ (\bar u_0+\bar u_1\zeta^ {- 1}) d\bar z\]
does also for all $\zeta\in\C ^*$. These forms define a family
\[
\nabla _\zeta =\nabla +\varphi _\zeta
\]
 of
flat connections in a trivial principal $ G_2 ^\C $ bundle $P $ over
$ M^2 $.

The complex Lie group $G_2 ^ \C $ is the connected component of the
subgroup of $ GL (7,\C) $ preserving a generic 3--form  $\alpha ' $
\cite {Bryant:87}. The action of the general linear group on the
space of 3--forms on a seven dimensional complex vector space $ V $
has a natural open orbit and by ``generic" we mean that $\alpha ' $
lies in this open orbit.  More precisely (see \cite {Hitchin:06})
the $\Lambda ^7V ^*$--valued symmetric bilinear form
\[
q (v, w): = \frac {-1}{6} (v\lrcorner\alpha ')\wedge
(w\lrcorner\alpha ')\wedge \alpha '
\] when viewed as a map $ V\rightarrow V ^*\otimes \Lambda ^ 7 V ^*$ has determinant $\kappa ^3 $, where $\kappa\in (\Lambda ^7V ^*) ^3 $ is a polynomial in $\alpha ' $. The open orbit is
given by $\kappa (\alpha ')\neq 0$.   One sees in this context that
$G_2 ^ \C\subset SO (7,\C) $ by showing that the connected component
of the subgroup preserving $\alpha ' $ also preserves the metric
defined by
\[
g =\frac {q} {\kappa ^ {1/3}}.
\]
Here we have taken the positive root with respect to the orientation
given by $\kappa ^3 $. When $\alpha ' $ has the usual normal form
\[
\alpha ' =(\theta _1\wedge\theta _2-\theta _3\wedge\theta
_4)\wedge\theta _ 5 + (\theta _1\wedge\theta _3 -\theta
_4\wedge\theta _2)\wedge \theta _6 + (\theta _ 1\wedge\theta _4
-\theta _ 2\wedge\theta _ 3)\wedge\theta _ 7+\theta _ 4\wedge\theta
_ 6\wedge\theta _7
\]
then $g $ is the standard Euclidean metric. Thus  we may
equivalently view the connections $\nabla _\zeta $ as acting on a
rank seven holomorphic vector bundle  { \boldmath { $ \mathcal  { V}
$} } on $M$ with $ G_2 ^\C $-structure given by a $ 3 $-form $\alpha
' $ on each $ V _ z $ with $\kappa (\alpha ')\neq 0 $. The family of
flat connections $\nabla _\zeta $ respects these symmetries in the
sense that
\[
\tau(\varphi_\zeta) =\varphi_ {\epsilon\zeta},\,\rho (\varphi_\zeta)
=\varphi_ {\bar\zeta ^ {- 1}},
\]
where $\epsilon =\exp {\pi\sqrt {- 1}/3} $.

To summarise, a non-isotropic almost-complex immersion $
f:\widetilde M\rightarrow S^6$ is equivalent to a $\C ^*$-family of
flat connections $\nabla +\varphi_\zeta $ in a trivial $ G_2^\C$
bundle { \boldmath { $ \mathcal  { V} $} } over $ M $, where
$\varphi_\zeta $ is symmetric with respect to $\tau,\rho $ as
specified above. Consider now the twisted loop algebrae given by
\[
{\Lambda ^ {\rho,\tau}_d\g_2 ^\C}:= \{ A_{\zeta} = \sum_{j=-d}^{d}
A_j \zeta ^j | A_j \in \g_j, A_{-j}=\rho(A_j), \zeta\in \C^*\} ,
\,d\in\mathbb {N}.
\]
These symmetries can also be expressed as
\begin {equation}\label {equation:symmetries}
\rho(A_{\zeta})=A_{ \bar{\zeta}^{-1} } \mbox { and }\tau
(A_{\zeta})=A_{\epsilon\zeta}.
\end {equation}
 If $A
=\sum_{j = - d} ^ {d}A_j\zeta ^ j: \widetilde M \rightarrow
{\Lambda^{\rho,\tau}_d\g_2 ^\C} $ satisfies the Lax pair equation
\[
dA =[A,\varphi_\zeta ]
\]
with $\varphi_\zeta = (A_{d -1} +A_d\zeta) dz+ (A_{1- d} + A_{-
d}\zeta^ {- 1})d\bar z $
 we call it a \emph{polynomial Killing field}. Note that the Lax equation is equivalent to
the requirement that $A $ be a parallel section of the pullback of
Ad$ (P) $ to $\widetilde M $. A polynomial killing field clearly
yields a superconformal almost-complex immersion $\widetilde
M\rightarrow S ^ 6 $, and if appropriate periodicity conditions are
satisfied than this immersion descends to $ M $. We shall now
restrict our attention to  almost-complex immersions  $ f: \C
\rightarrow S ^ 6 $ obtained from polynomial Killing fields; these
maps are said to be \emph{of finite type}, and include all doubly
periodic examples \cite {BPW:95}. The Lax equation clearly forces
$d=6k+1$.

\section{Spectral data for the Lax equations}
We show that a non-isotropic almost-complex immersion $ f:
\C\rightarrow S ^ 6 $ of finite type gives rise to an algebraic
curve, which we call the spectral curve.  There are a number of
different constructions of spectral curves which have been used in
other geometric settings.  Beginning with a commutative algebra of
polynomial Killing fields, one can use the characteristic
polynomial, as in \cite {AM:80:1, Griffiths:85, Hitchin:06}, the
eigenline curve, as in \cite {Hitchin:90} or take Spec, as in \cite
{McIntosh:95}.  Alternatively, one could take a more Lie-theoretical
approach, as in \cite {KP:94, Kanev:89}.  The characteristic
polynomial is the simplest definition, but can for some surface
classes have unnecessary singularities which do not reflect the
underlying geometry.  We shall prove that for non-isotropic
almost-complex immersions of the plane into $ S ^ 6 $, the
characteristic polynomial is generically smooth, justifying our
choice.

We may pull the trivial bundle { \boldmath { $ \mathcal  { V} $} }
and the three form $\alpha ' $ back to $\C\times\P ^ 1 $, and write
$ \mathbf { V},\alpha ' $ for the restrictions to $\{z\}\times\P^1 $
(so $\mathbf{V} $ is the trivial rank seven holomorphic vector
bundle on $\P^1 $). The polynomial Killing fields $ A '_{\zeta} (z)
\in\loop $ are holomorphic sections of $\co(2(6k+1)) \otimes \End
\mathbf{V}$. We shall realise the characteristic polynomial of $ A
'_{\zeta} (z) $ inside the total space of a line bundle over $\P ^ 1
$. The eigenvalues of an element of $\g_2^\C $ are of the form $0$,
$\mu _1$, $\mu_2$, $\mu_3$, $-\mu_1$, $-\mu_2$, $-\mu_3$, and
satisfy $\mu_1+\mu_2+\mu_3=0$. Hence the characteristic polynomial
has the form
\[
\det (\mu \cdot id_{\mathbf{V}} -  A '_{\zeta} (z)) = \mu (\mu^6 -
a_1 \mu^4 + \frac{a_1^2}{4} \mu^2 - a_2 )
\]
with
\[
 a_{1}(\zeta) = \mu_1^2+\mu_2^2+\mu_3^2, \quad a_{2}(\zeta) = (\mu_1 \mu_2  \mu_3)^2 ,
\]
We let $\hat{X}$ be the complex surface which is the total space of
the line bundle $\co(2(6k+1))$ and $\pi_{\hat{\sx}}$ its projection
onto $\Pj^{1}$. The pull-back of $\co(2(6k+1))$ to $\hat{X}$ has a
tautological section $\mu$.  The characteristic polynomial is thus a
section of $ H^{0}(\hat{X}, \pi_{\hat{\sx}}^{\ast} \co(14(6k+1)) )
$, and defines a reducible algebraic curve $\hat\Sigma '  $ in $\hat
X $, one of whose components is $\Pj^{1}$.  We denote the other
component by $\hat {\Sigma}  $. The matrices $ A ' _ {\zeta} (z) $
for different $ z $ are related by conjugation, and hence
$\hat\Sigma ' $ is independent of $ z $.

  The $\tau $-symmetry of \eqref {equation:symmetries} gives that $a_j (\epsilon\zeta)
=a_j(\zeta) $ so we may write
$a_{j}(\zeta)=b_{j}(\la)|_{\la=\zeta^6}$.   Notice that this forces
the degrees of $a _ 1, a _ 2 $ to be divisible by six so that $a _ 1
$ cannot be an element of $H^{0}(\Pj _\zeta ^{1}, \co(4 (6k+1)) ) $
of full degree;
we have instead that $a _ 1\in H^{0}(\Pj _\zeta ^{1}, \co(4 (6k)) )
$. Clearly $a _2\in H^{0}(\Pj _\zeta ^{1}, \co (12(6k+1))$. The
$\rho $-symmetry of  \eqref {equation:symmetries}  yields
$\overline{b_{j}(\la)}=b_{j}(\bar{\la}^{-1})$. Let $ X $ denote the
total space of the line bundle $\co (2 ( k +1) ) $ and $\pi_{\sx}$
its projection onto $\Pj^{1}$. Writing $\eta $ for the tautological
section of $\pi_{\sx} ^*\co (2 ( k +1) ) $ over $X $, then again
$\Sigma ':=\hat{\Sigma '} /\tau $ decomposes into a copy of $\P^1 $
and the algebraic curve $\Sigma $ given by the divisor of
 \begin {equation*}
\eta^6-b_1\eta^4+\frac{b_1^2}{4} \eta^2 - b_2 \in H^{0}(X,
\pi_{\sx}^{\ast} \co(12 (k +1 )) ).
\end {equation*}
We call $\Sigma '$ the \emph{spectral curve} of the almost-complex
immersion $ f:\C\rightarrow S ^ 6 $, and we term $\Sigma $ the
\emph{main component} of $\Sigma '$.

\begin{theorem}
The main component $\Sigma $ of a generic spectral curve $\Sigma ' $
is smooth, with genus $5(6k+1)$. The real dimension of the moduli of
such curves $\Sigma $ is $(16k+4)$.
\end{theorem}
\begin{proof}   We begin by proving smoothness, for which we use an argument similar to that employed by \cite {KP:94} in their study of spectral curves of $G _2 $ Higgs bundles.
Set $B:=H^{0}(\Pj^{1}_{\la}, \co(4k)) \oplus H^{0}(\Pj^{1}_{\la},
\co(12k+2)) $, and write $ B _\R $ for the real slice given by
$\overline{b_{j}(\la)}=b_{j}(\bar{\la}^{-1})$.  The main component
$\Sigma $ of the spectral curve is given by $ (b _ 1, b _ 2) \in B
_\R $, and we wish to show that those points which yield singular
spectral curves are contained in a lower dimensional subvariety of $
B _\R $.   Recall that  the complete linear system $|D| $ defined by
a divisor $D$ on $ X $ is the space of all effective divisors
linearly equivalent to $D$, a linear system $l $ is any linear
subspace of $|D|=\P H^0(X,\mathcal O (D)) $ and that $p\in X $ is a
base point of $l$ if $p $ is contained in all members of $ l$. We
denote the base locus of $ l $ on $ X $ by $B_X{l} $.  Bertini's
theorem \cite {Bertini:1907} says in part that a generic element of
a linear system is smooth away from the base locus of the system.
 For each fixed $b_1\in H^0(\P^1,\co (4k)) $,  as $c\in\R$ and $b_2\in\co (12 k+2) $ vary,
\begin {equation}\label{equation:linear}
c(\eta^6-b_1\eta^4+\frac{b_1^2}{4} \eta^2 )- b_2
\end {equation}
 defines a linear subsystem $l (b_1 ) $ of $|\pi_X^*\co(12 k+2 )| $.  This contains the affine subspace $ a (b_1) \simeq\P (\pi ^*H^0(\P^1,\co (12 k+2)) $ defined by setting $c=1$.  We will show that $ a (b_1) $ is base point free, and hence the same is true for $ l (b _1) $. In particular $ a (b_1) $ contains the divisor of  $\eta^6-b_1\eta^4+(b_1^2/4) \eta^2 = \eta^2
(\eta^2-b_1 / 2)^2$ and so the base locus $B_X{a(b_1 ) } $ of $a(b_1
)  $ is a subset of
\[
B_{ \{\eta=0\} } a (b_1) |_{ \{\eta=0\} }\cup B_{
\{\eta^{2}=b_{1}/2\} }a (b_1) |_{ \{ \eta^{2}=b_{1}/2 \} }.
\]
The restriction of $ a (b_1) $ to the projective line $\{\eta = 0\}
$ clearly has no base points.  The hyperelliptic curve $ C ({b _1})
$ defined by $\eta^{2}=b_{1}/2 $ is smooth for a generic choice of $
b _1 $, and writing $\pi $ for the projection to $\P ^1 $, we have
\[
 B _ {C (b _1)} \P (\pi ^*H^0(\P^1,\co (12 k+2))
 = B _ {C(b _1)} \pi ^*|\co (12 k+2)|.
\]
But the projection of this base locus to $\P ^1 $ is then contained
in the base locus for $|\co (12 k+2)| $, which is empty.

 We have then that for a generic $ b _ 1 $ the linear system $a (b_1)$ has no base points.  Since dividing \eqref{equation:linear} by the constant $ c$ does not change the resulting curve,  Bertini's theorem implies that a generic $ (b _ 1, b _ 2)\in B $ defines a smooth curve.  Taking real points yields the analogous result for $ B _\R $, which has real dimension $ (4 k+1) + (12k +3) = 16 k +4 $, and parameterises the space of spectral curves.

We now compute the degree of the ramification divisor $ R $ of the
projection $\lambda:\Sigma\rightarrow\P ^ 1 $.  The cover is clearly
totally branched over $0$ and $\infty$, and the ramification points
for $\lambda\in\C ^*$ are given by the vanishing of the discriminant
\[
\triangle= b_2(\frac{1}{2} b_1^3 -27b_2).
\]
At the $ 12 k +2 $ points for which $ b _ 2 (\lambda) = 0 $, the
equation for $\Sigma $ factors as
\[
\eta ^ 2 (\eta ^ 2 -\frac {b _ 1} {2}) ^ 2 = 0,
\]
and we may assume that three pairs of eigenvalues come together as
$\mu_{1}=-\mu_{1}=0$, $\mu_{2}=-\mu_{3}=\sqrt{b_{1}/2}$ and
$\mu_{3}=-\mu_{2}=-\sqrt{b_{1}/2}$.  At the $ 12 k +2 $ points where
$27b_2-(1/2) b_1^3=0$, we have
\[
\left (\eta ^ 2 -\frac {b _ 1} {6}\right) ^ 2\left (\eta ^ 2 -\frac
{2 b _ 1} {3}\right)= 0
\]
and we may assume that two pairs of eigenvalues come together as
$\mu_{1}=\mu_{2}=\sqrt{b_{1}/6}$, $-\mu_{1}=-\mu_{2}$ (thus $\mu
_{3}=-2 \sqrt{b_{1}/6}$). Hence
\[
|R| = (12k+2)\cdot 3+(12k+2)\cdot 2+2\cdot 5 = 20 (3k + 1)
\]
so by the Riemann-Hurwitz formula,
\[
2 - 2g = 6 \cdot 2 - |R|= -60 k-8
\]
and the genus of the smooth algebraic curve $\Sigma $ is $g =
5(6k+1)$. Since arithmetic genus is constant in families, this is
also the arithmetic genus of any singular curves $\Sigma $.
\end{proof}
\begin {remark} \label {degree genus} We hence also obtain that generically the curve $\hat\Sigma $ is smooth, the ramification divisor $\hat R $ of the map $\zeta:\hat\Sigma\rightarrow\P ^1 $ has degree
$ |\hat R| = 60\degree + 10 $ and the genus of $\hat\Sigma $ is $
30\degree $.
\end {remark}
We explain now how the spectral curve $\Sigma '$ can be realised as
the characteristic polynomial of a polynomial killing field
depending only on $\lambda =\zeta ^6$.  Let
\[
C _\zeta =\mbox {diagonal} (1, S_\zeta, S_{\zeta ^2}, S_{\zeta ^3})
\]
where
\[
S_\zeta =\left (\begin {array} {cc}\frac 12 (\zeta +\zeta^ {- 1}) &\frac {-1}{2i} (\zeta -\zeta^ {- 1})\\[4pt ]
\frac {1}{2i} (\zeta -\zeta^ {- 1}) &\frac {-1}{2} (\zeta +\zeta^ {-
1}) \end {array}\right)
\]
and define
\[
A '_\zeta (z) =\mbox {Ad}_{C_{\zeta^ {- 1}}}  A ' _\zeta (z).
\]
\begin {lemma}
For each $ z\in T^2$, the endomorphism $A ' _\zeta (z) $ is a
polynomial Killing field depending only on $\lambda =\zeta ^6 $.  As
a function of $\lambda $ it satisfies $A ' _\lambda (z)\in\Lambda _
{2(k+1)} ^\rho\g _2 ^\C $ whenever $  A ' _\zeta (z)\in\Lambda _
{2(6k + 1)} ^ {\rho,\tau}\g _2 ^\C $.
\end {lemma}
\begin {proof} By the rigidity of holomorphic functions, it suffices to prove these results for $\zeta = e ^ {i\theta}\in S^1 $.  Then $S _\zeta = R _\theta $ and so $A ' _ {\epsilon\zeta} =\mbox {Ad}_{C _ {\epsilon^ {- 1}\zeta^ {- 1}}}\mbox {Ad}_{C _ {\epsilon}} A ' _\zeta =A ' _\zeta $, and so $A ' _\zeta $ depends only on $\lambda =\zeta ^6 $. Also $ S _\zeta ^ {- 1} = R _\theta ^ {- 1} = R _ {-\theta} = S _ {\zeta^ {- 1}} $ and so $A '_\lambda $ is polynomial in $\lambda $. Since $ C _\zeta $ and $ C _ {\zeta^ {- 1}} $ are each of the form $\sum _ {j=-3}^3 C_j\zeta ^ j $, {\it a priori } we would expect that $ A ' _\zeta \in \co(12k+14)\otimes \End \mathbf{V}$.  However we have just shown that $ A ' _\zeta $ is a function of $ \lambda =\zeta^6 $ so we must have $ A ' _\zeta \in \co(12 (k+1) )\otimes \End \mathbf{V}$ or $A '_\lambda\in \co(2 (k+1) )\otimes \End \mathbf{V}$.  The $\rho $-symmetry is inherited from $ A ' _\zeta $.
\end {proof}

 \section {Symmetries of the Linear Eigenline Flow }
We now identify geometric structures on the spectral curve resulting
from the fact that the immersion$ f $ is almost-complex  and show
that the immersion gives rise to a linear flow in the intersection
of two Prymians of $\Sigma $.   In his study of $G_2$ Higgs bundles
\cite {Hitchin:06}, Hitchin encoded the $ G_2$ structure of that
problem in terms of spectral curve data and in particular in terms
of a linear flow in a similar sub-torus of the Jacobian.  Our
approach here is an adaption of his work, and can be viewed as
another aspect of the similarity between the spectral curve
approaches to finite-type harmonic maps and to Higgs bundles.

We shall eventually express these symmetries in terms of the
spectral curve $\Sigma '$ and its main component $\Sigma $.  However
not all of the geometric structures that we shall utilise (for
example the symplectic form introduced below) naturally descend to
the quotient picture and so we shall first work with the
unquotiented curves $\hat\Sigma '$ and $\hat\Sigma $.  Essentially
this is necessitated by the fact that the degree of the polynomial
killing field $A_{\zeta}(z) $ is not in general divisible by six and
so an attempt to define a symplectic form as below using $A _\lambda
(z) $ would not yield an everywhere non-degenerate form. As
explained earlier, the generic 3--form $\alpha '$ on the trivial
rank seven holomorphic vector bundle $ \mathbf V $ defines a
non-degenerate symmetric bilinear form $g $, and we will now
investigate the symmetries on the eigenline bundles resulting from
this special orthogonal structure.  The form $g$ induces on $
\mathbf V $ for each $ z\in\C $ a skew-symmetric bilinear form
$\omega_z^{'} $ valued in $\mathcal O (2\degree) $
\begin {equation}\label {equation:symplectic form}
\omega_z^{'} (v_1, v_2) = g (A ^{'}_ {\zeta} (z) v_1, v_2).
\end {equation}
Let $\mathbf V _{0} =\mathbf V _{0} ^ {z} $ be the line bundle
contained in $\ker A^{'}_{\zeta} (z) $ and define $\mathbf V_{1}
=\mathbf V_{1} ^ {z} $ as the quotient
\[
0 \to \mathbf V_{0} \to \mathbf V \to \mathbf V_{1} \to 0 .
\]
Here and henceforth we suppress the dependence on $ z $; we shall
show that the holomorphic line bundles $ \mathbf V_0 ^ z$ are all
isomorphic.

The forms $g $ and $\omega _ z $ are clearly well-defined on the
quotient $\mathbf V _1 $.  The dual bundle $\mathbf V _1 ^*$ is the
annihilator of $\mathbf V _0 $, and writing
\[
 \gamma: \mathbf V ^*\simeq \mathbf V
\]
for the isomorphism given by $ g $, then for each $\zeta\in\P ^1 $,
the subspace $\gamma (\mathbf V _1 ^*(\zeta)) $ is the orthogonal
complement of $ \mathbf V _0 (\zeta) $.
Away from the finitely many $\zeta \in\P ^ 1 $ satisfying $ {a}_ 2
(\zeta) =(\mu _ 1\mu _ 2\mu _ 3) ^ 2 = 0 $, the restriction $A_\zeta
(z) $ of the polynomial killing field to the orthogonal complement
of $\mathbf V_0 $ has no kernel.  Setting $\mathbf{E} = \mathbf V _1
^ z\otimes \mathcal O (-\degree )$, the obvious analog of \eqref
{equation:symplectic form} gives a skew-symmetric form $\omega _ z$
on $\mathbf{E _ z}  $, clearly non-degenerate when ${a} _2
(\zeta)\neq 0 $.
\begin {lemma} [c.f. \cite {Hitchin:06} ]
For each $z\in\C $, the form $\omega _ z $ is everywhere
non-degenerate, so the vector bundle $\mathbf{E _ z} $ is symplectic
and hence holomorphically trivial.  The restriction $A_\zeta (z) $
of the polynomial killing field $A '_\zeta (z) $ acts as a
symplectic endomorphism of $\mathbf{E _ z} $.
\end {lemma}
  In light of this, we shall again suppress the $z$ and simply write $\mathbf{E}$ since holomorphically all these bundles are isomorphic.
\begin {proof}
The polynomial killing field $A _ {\zeta}(z) $ is a section of
$\mathcal O (2\degree)\otimes \g_2 ^\C $ and so $\omega _ z\in H ^ 0
(\Lambda ^ 2 \mathbf V ^*\otimes \mathcal O (2\degree)) $.  Thus
$\omega _z ^ 3 $ is a section of $\Lambda ^ 6 \mathbf V ^* \otimes
\mathcal O (6\degree ) $.
  Furthermore, the metric $g $ gives an isomorphism
\[
\psi: \Lambda ^ 6 \mathbf V ^*\rightarrow  \mathbf V
\]
characterised by $\alpha '=6\psi (\alpha ^ {'})\lrcorner\mbox {vol}
$, where vol denotes the volume form.
  A straightforward computation yields
\begin {equation}
\omega _ z (A _ {\zeta}  (z) v _ 1, v _ 2) +\omega _z (v _ 1, A _
{\zeta}  (z) v _ 2) = 0\label {equation:symplectic}
\end {equation}
and from this we see that
 under the induced action on $\Lambda ^ 2 \mathbf V ^*$ the skew-symmetric form $\omega _z $ satisfies $A _ {\zeta} (z) (\omega _ z) = 0 $. Hence for each $z$ the section of $\mathbf V $ defined by
\[
v _0 ^ z = \psi (\omega _ z ^3)
\]
is a zero--eigenvector of $A _ {\zeta} (z) $. Let $ e_0,\ldots , e
_6 $ be an orthonormal basis for $ H ^ 0(\P^1, \mathbf V) $, with $
e_0 $ valued in $ \mathbf V _ 0 $.  Then
\[
v _ 0 (z) = - i\mu _ 1\mu _ 2\mu _3e _ 0
\]
and we see that when regarded as a section of  $H ^ 0 (\P^1,\mathbf
V _0\otimes \mathcal O (6 \degree) ) $, the zero-eigenvector section
$v _0 ^ z $ is nowhere--vanishing, so defines an isomorphism
$\mathcal O (- 6 \degree)\simeq\mathbf V _0\subset \mathbf V $. The
isomorphism $\psi $ gives $\Lambda ^ 6 (\mathbf V _ 1)\simeq \mathbf
V _ 0 ^*$ and we have
\[
\Lambda ^ 6 (\mathbf V _ 1)\simeq \mathbf V _ 0 ^*\simeq \mathcal O
(6\degree)
\]
and thus
\[
\Lambda ^6\mathbf{E} \simeq\co.
\]
The form $\omega _ z $ must then be
 non-degenerate everywhere and so $\mathbf{E} $ is a symplectic vector bundle. The symplectic form $\omega _ z $ defines an isomorphism $\mathbf{E}\simeq \mathbf{E} ^*$ and holomorphically $\mathbf{E}$ is the trivial rank 6  bundle. From (\ref {equation:symplectic}) we see that $A _ {\zeta}  (z) $ acts as a symplectic endomorphism of $\mathbf{E} $.
\end {proof}
When $a _2 (\zeta)\neq 0 $, $\omega _ z $ provides an isomorphism $
\mathbf V _ 1 ^*\simeq \mathbf V _1 $, which allows us to identify
$\mathbf V _1 $ with the orthogonal complement $\omega _ z (\mathbf
V _ 1 ^*) $ of $ \mathbf V _0 $ at such points.  For such $\zeta $,
the restriction of the polynomial killing field $A '_\zeta (z) $ to
$\omega _ z (\mathbf V _ 1 (\zeta))\subset \mathbf V (\zeta) $
agrees with its action on $ \mathbf V _1 (\zeta) $ as a quotient,
and (\ref {equation:symplectic}) similarly still holds.  For $\zeta
$ satisfying ${a} _2 (\zeta) = 0 $, $\mathbf V _0 (\zeta ) $ is
contained in its orthogonal complement $\gamma (\mathbf V _1 ^*) $.

\begin {definition}
For each $ z\in\C $ define ${\hat \ce} _ z \rightarrow{\Sigma} $ to
be the unique line bundle contained
in the sub-sheaf $\ker(\mu\cdot id
-\zeta ^* A_{\zeta})  (z) \subset \zeta ^*\mathbf {E} $. We call
$\hat \ce _ z $ the {\emph{eigenline bundles}}.
\end {definition}
 We note that one
could instead look at the eigenlines of the action of $A' _\zeta $
on $\zeta ^* \mathbf {V} $ and indeed to reconstruct the original
almost-complex immersion $ f $ we will need to consider glueing
information on the points of $\hat\Sigma' $ where the two components
intersect, given by $a_2(\zeta) =0 $ (or their equivalent on the
quotient curve $\Sigma' $).  For now though we are interested in the
symmetries of the bundles $\hat\ce _ z $. The eigenline bundles
satisfy
\[
\zeta_*\hat \ce ^*_z \simeq \mathbf{E } ^*\simeq \mathbf{E }.
\]
We have on $\hat\Sigma $ the involution
\[
\sigma: (\zeta,\mu)\mapsto (\zeta, -\mu),
\]
which defines a Prym variety in which, as we shall see, our
eigenline bundles lie.
  Let $ \hat C _ 1\simeq \hat\Sigma/\sigma $ be the curve defined by
\[
y ^ 3 - {a}_1 (\zeta) y ^ 2+\frac {{a}_1 (\zeta) ^ 2} {4} y - {a}_2
(\zeta) = 0
\]
and write $\hat\pi _ 1 (\mu,\zeta) = (\mu ^ 2,\zeta) $ for the
natural projection $\hat\Sigma\rightarrow \hat C _ 1 $.  Define
\[
\begin {array} {rccc}\mbox {Nm}_1 &\mbox { Pic} (\hat\Sigma) &\rightarrow &\mbox {Pic} (\hat C_1) \\
&\mathcal O (D) &\mapsto & \mathcal O(\hat\pi _1(D)).\end {array}
\]
The  Prym variety $P(\hat\Sigma, \hat C _ 1) $ is defined to be the
connected component of ker\nolinebreak$ {(\mbox {Nm}_ 1)} _*$
containing the identity. When $ \hat\pi _1^*$ is injective, $
{(\mbox {Nm}_ 1)} _*$ is given by the dual of $\hat\pi _1 $ and so
ker\nolinebreak$ {(\mbox {Nm}_ 1)} _*$ is itself connected \cite
{Mumford:74}.  The criterion given in \cite {BNR:89} [ 3.10 ] for
the injectivity  of $\hat\pi _1 ^*$ is here easily seen to be
satisfied.
\begin {theorem}\label {first Prym}  The eigenline bundles $\hat\ce _ z $ have degree $ - 30\degree -5 $
and lie in a constant translate of the Prym variety defined by
$\hat\pi _1:\Sigma\rightarrow \hat C _1 $:
\[
\hat{\mathcal E} _z\otimes \mathcal O(\frac 12\hat R)\in P(\Sigma,
\hat C _1).
\]
They also satisfy
\[
\tau ^*\hat{ \ce} _ z\simeq \hat{ \ce} _ z
\]
and the reality condition
\[
\overline {\rho ^*\hat {\ce}_z}\simeq\hat \ce_z.
\]
\end {theorem}
\begin {proof}
The push-forward of the sheaf of sections of each eigenline bundle
$\zeta _*\co (\hat\ce _ z ) $ is the sheaf of sections of a vector
bundle \cite {Gunning:67}, and the fibres of this bundle are
\[
(\zeta _*\hat\ce _ z ) _\zeta =\oplus _ {(\zeta,\mu)\in\Sigma}\; J _
{r} (\hat \ce _ z) _ {(\zeta,\mu)},
\]
where $J _ {r} (\hat \ce _ z) {(\zeta,\mu)}$ denotes the jet space
of local sections of $\hat \ce _ z $ of degree $r=\mbox {deg} \hat R
_ {(\zeta,\mu)} $ and $ \hat R $ is the ramification divisor of the
cover $\zeta:\hat\Sigma\rightarrow\P ^1 $. Note that $J _ {0} (\hat
\ce _ z) {(\zeta,\mu)}$  is simply $\hat \ce _ z {(\zeta,\mu)}$. The
natural projection of the jet space to the bundle fibre gives
\[
0\rightarrow\oplus _ {(\zeta,\mu)\in\Sigma}\; J _ {r} {({\hat \ce _
z})} _ {(\zeta,\mu)} /{({\hat \ce _ z})} _ {(\zeta,\mu)}
\rightarrow {(\zeta _*\hat\ce _ z)} _ {(\zeta,\mu)}\rightarrow
\oplus _ {(\zeta,\mu)\in\Sigma}\; \hat {\ce _ z} _
{(\zeta,\mu)}\rightarrow 0.
\]
Dualising, we have the exact sequence of sheaves (c.f.\cite
{Hitchin:87:Higgs})
\begin {equation}\label {push forward}
0\rightarrow\co (E ^*)\rightarrow\co (\zeta _* \ce _ z)
^*\rightarrow\mathcal R\rightarrow 0,
\end {equation}
where $\hat{\mathcal R }$ is the skyscraper sheaf with stalk $\oplus
_ {(\zeta,\mu)\in\Sigma}\; J _ {r} (\hat {{\ce} _ z) }^*_
{(\zeta,\mu)}/{\hat {{\ce} _ z }^*}_ {(\zeta,\mu)} $ on the
ramification divisor of $\zeta $.  The degree of the push-forward
bundle is given by
\[
d (\zeta _*\hat\ce _ z) =d (\hat\ce _ z) + 1 - g _ {\hat\Sigma}
-\mbox {deg} (\zeta) (1 - g _ {\P^1}) =d (\hat\ce _ z) - (30\degree
+5),
\]
where the first equality is a standard result (see e.g.\cite
{Hitchin:99}) and the second uses Remark~\ref {degree genus}.  In
the same remark we computed that the degree of the ramification
divisor $ R $ is $ 60\degree + 10 $, so using the exactness of
\eqref {push forward} we obtain $d(\hat\ce_z)=-30(\degree)-5 $ as
claimed. Note that we could also have used the relation $\zeta_*\hat
\ce ^*_z \simeq \mathbf{E } ^*\simeq \mathbf{E } $ to compute the
degree of the eigenline bundles.

 Under the involution $\sigma $ the eigenline bundles satisfy
\[
\sigma ^*\hat \ce _ z (\zeta,\mu) =\hat \ce _ z (\zeta, -\mu)
\]
and by definition $\omega _z $ defines a non-trivial section of
\[
\hat \ce _ z ^*\otimes \sigma ^*\hat \ce _ z ^*\otimes \mathcal O
(\hat R).
\]
Since we have shown that this line bundle has zero degree we
conclude that it must be the trivial bundle. This shows that the
bundles $\hat \ce _ z ^*\otimes \holomorphic (-\frac 12\hat R) $ lie
in the Prym variety defined by the involution $\sigma $, where we
have fixed a square root of the bundle $\holomorphic (- \hat R) $.
Using \eqref {equation:symmetries}, the symmetry of the eigenline
bundles with respect to $\rho $ and $\tau$ is clear.
\end {proof}

\begin {comment}  arising from the symplectic forms $\omega _ z $ to show that the eigenline flow to a sub-torus of the Jacobian.
The eigenline bundles also satisfy a further symmetry which together
with $\sigma $ encodes the fact that the frame to our almost-complex
torus lies in $G_2\subset SO(7)$. These correspond to the two types
of symmetry of the eigenvalues  $\mu _1$, $\mu_2$, $\mu_3$,
$-\mu_1$, $-\mu_2$, $-\mu_3$ of $A _ {\zeta} (z)| _ {\mathbf V _ 1}
$ by
\begin {enumerate}
\item the fact that each eigenvalue occurs along with its negative
\item the relation
$\mu_1+\mu_2+\mu_3=0$.
\end {enumerate}
We have encoded the first type of symmetry above using the
involution $\sigma $, and shown that it results in the eigenline
bundles lying in a translate of a certain Prym variety.

\end {comment}
So far we have utilised only the involution $\sigma $ which could be
described purely in terms of the symplectic structure.  To encode
the fact that the frame to our almost-complex torus lies in
$G_2\subset SO(7)$ we return our attention to the 3--form $\alpha
'\in H ^ 0(\P^1, \mathbf V) $ satisfying $\kappa(\alpha ^ {'})\neq 0
$ which defines the $ G_2 $--structure on $ \mathbf V $.

 Denote by $\alpha  $ the restriction of $\alpha '$ to the orthogonal complement $\mathbf V _0 ^\perp\simeq\gamma ( \mathbf V _1 ^*) $.of $ \mathbf V _0 $.
We defined earlier the symplectic vector bundle $ \mathbf{E} =
\mathbf V _1\otimes \mathcal O (-\degree) $,
and so we may write $\alpha \in H ^0 (\P ^1, {\Lambda ^3} \mathbf{E}
^*\otimes  \mathcal O(3\degree) $.  Define as in \cite {Hitchin:06,
Hitchin:00}
\[\begin {array} {rccc}
K _ {\alpha }: & \mathbf{E} &\rightarrow & \mathbf{E}\otimes \mathcal O(6\degree)\\
&v &\mapsto & (v\lrcorner\alpha )\wedge\alpha \end {array}
\]
where we have used 
$\Lambda ^5\mathbf{E}^*\simeq \mathbf{E}$.  It is shown in \cite {Hitchin:00} that the entire space 
$\mathbf{E}$ is an eigenspace of $K _ {\alpha } ^2 $, and we write
the eigenvalue as
\[
s (\alpha ): =\frac 16\mbox {trace}K _ {\alpha } ^ 2\in H ^ 0 (\P ^ 1 ,  \mathcal O(12\degree).  
\]

Furthermore, if ${a}_2 (\zeta) =0 $, then from the definition of $g
$, the vector $ v_0\in \mathbf V_0^\perp\simeq \mathbf V_1^* $
satisfies $(v_0\lrcorner\alpha ^ {'})\wedge (v_0\lrcorner\alpha ^
{'})\wedge\alpha ^ {'}=0$ and so the eigenvalue  $s (\alpha )
(\zeta) $ must vanish.  Hence ${a}_2$ divides $s (\alpha ) $ and the
eigenvalues $\sqrt{s(\alpha )} $ are well-defined on the
hyperelliptic curve $\hat C_2 $ given in the total space of
$\mathcal O ( -6\degree) $ by
\[
z^2= {a} _ 2 (\zeta).
\]
Since ${a}_2(\zeta) $ has degree $ 12 (\degree) $, the curve $ \hat
C _ 2 $ has genus $ 36  k+10 $.  Write
\[
\hat{p} _2:\hat C _2\rightarrow\P^1
\]
for the natural projection and observe that
\[
\begin {array} {rccc}
\hat\pi _2: & \hat\Sigma &\rightarrow &\hat C_2\\
& (\zeta,\mu) &\mapsto & (\zeta,\mu (\mu -\frac {{a}_1}{2}))\end
{array}
\]
exhibits $\hat\Sigma $ as a three-to-one cover of $\hat C_2$.
\begin {theorem}\label {second Prym}
The constant translate $\hat {\mathcal E}_z\otimes \mathcal O(\frac
12 \hat R) $ of the eigenline bundles also lies in the Prym variety
$P(\hat\Sigma,\hat C_2)$ defined by $\hat\pi_2:\hat\Sigma\rightarrow
\hat C_2$.\end {theorem}
\begin {proof}
 The map $ p_2^*K _ {\alpha } $ has for each $\zeta $ two 3-dimensional eigenspaces which define vector bundles $ \mathbf{E} ^ + $ and $ \mathbf{E} ^ - $ on $ \hat C _ 2 $. If $ W $ is any six-dimensional complex vector space,  $ GL (6) $ acts on $\Lambda ^3 W ^*$ with stabiliser $ SL (3)\times SL (3)\times\Z _2 $.   It is shown in \cite {Hitchin:00} that if we choose a basis $v_1,\cdots ,v_6 $ for $W $ and write the dual basis as $\theta _1,\cdots ,\theta _6 $ then the three forms $\alpha  $  satisfying $s (\alpha ) \neq 0 $ have normal form
\[
\alpha  =\theta _ 1\wedge\theta _ 2\wedge\theta _ 3 +\theta _
4\wedge\theta _ 5\wedge\theta _ 6,
\] and that
\[
K _ {\alpha } v _i=v_i  (\theta _ 1\wedge\ldots \wedge\theta
_6),\,i=1,2,3;\;K _ {\alpha } v _i= - v_i  (\theta _ 1\wedge\ldots
\wedge\theta _6),\,i=4,5,6.
\]
We label our eigenspaces so that away from the divisor $D$ on $\P ^
1 $ given by  ${a}_2(\zeta) = 0 $, the span of $v_1,v_2,v_3$ in this
normal form is $ \mathbf{E} ^ + $.  The restriction of $\alpha  $ to
$ \mathbf{E} ^ + $ is thus non-vanishing away from $\hat{p} _2 ^*(D)
$. Returning momentarily to the vector space situation, the
hypersurface within the space of three forms  $\Lambda ^3 W ^*$
satisfying $s =0 $ itself has an open orbit under the action of $ GL
(6) $, in which the three-form has normal form \cite {Hitchin:06}
\[
\alpha  =\theta _1\wedge\theta _5\wedge\theta _6 +\theta
_2\wedge\theta _6\wedge\theta _4 +\theta _3\wedge\theta
_4\wedge\theta _5.
\]
With respect to both of these normal forms, as in \cite {Hitchin:06}
the symplectic form $\omega _ z =\theta _ 1\wedge\theta _4 +\theta _
2\wedge\theta _ 5 +\theta _ 3\wedge\theta _6 $ and the eigenspaces $
E ^ + $, $ E ^ - $ are Lagrangian.

A straightforward calculation shows that the endomorphisms $K _
{\alpha } $ and $A _ {\zeta}  $ commute, and hence $\hat{p} _2 ^* A
_ {\zeta}  $ restricts to $A _ {\zeta} ^ {+}\in\mbox {End}
(\mathbf{E} ^ +)\otimes \hat{p} _2 ^*\holomorphic (2\degree) $. The
curve $\hat\Sigma $ is then biholomorphic to the curve given in the
total space of $H^0(\hat C_2,\holomorphic (2\degree) $ by the
characteristic polynomial of $A _ {\zeta} ^ {+} $, and the eigenline
bundles $ \hat {\mathcal E}_z $ are also given by the eigenlines of
the action of $A _ {\zeta} ^ {+} $ on $\mathbf {E} ^ + $.  Then
using also that $\mathbf{E} ^ + $ is Lagrangian,
\begin {equation}\label {eq:pushforward}
(\hat\pi_2)_* \mathcal E^*_z= {(\mathbf{E} ^ +)} ^*\simeq\rho ^*
\mathbf{E}/\mathbf{E} ^ +.
\end {equation}
Thus $ \mathbf{E} ^ + $ is the kernel of the natural surjection
$\hat{p} _2 ^*\mathbf{E}\rightarrow (\hat\pi _ 2) _*\hat {\mathcal
E}^*_ z\simeq ({ \mathbf{E} ^ +}) ^* $.

For a generic almost-complex immersion $ f: T ^ 2\rightarrow S ^ 6
$, at all points $\zeta\in D $ the three-form $\alpha  (\zeta) $
will lie in the open orbit described above.  We assume then that on
$\hat{p} _2 ^*(D)  $, the three form $\alpha  (\zeta) $ lies in this
open orbit, and then as in \cite {Hitchin:06} using the normal form
one sees that the restriction of $\alpha  $ to $ \mathbf{E} ^ + $
vanishes on $\hat{p} _2 ^*(D)  $ with multiplicity two.  Thus
$\alpha  $ defines an isomorphism
\begin {equation}\label {eq:determinant}
\Lambda ^3(\mathbf{E} ^ +)\simeq \hat{p} _2 ^* \mathcal O
(-9\degree).
\end {equation}

Once again, a straightforward calculation gives that for any line
bundle $\mathcal L $ on $\hat\Sigma $,
\begin {equation}\label {eq:norm}
\mbox {Nm} _2 \mathcal L=\det ({\hat{\pi _2}} _*\mathcal L)\otimes
\det ({\hat{\pi _2}} _*(\mathcal O (\hat R _2)),
\end {equation}
where now $ \hat R _2 $ denotes the ramification divisor of $\hat\pi
_2 $. Then from \eqref {eq:pushforward}, \eqref {eq:determinant} and
\eqref{eq:norm} we have that
\begin{eqnarray*}
\mbox {Nm} _2 ( \hat {\mathcal E}_ z ^*\otimes \holomorphic (-\frac
12\hat R)) & = &\hat{p} _2 ^*
 \mathcal O (9\degree)\otimes \det ({\hat{\pi _2}} _*(\mathcal O (\hat R _2))\otimes\mbox {Nm} _2 (\holomorphic (-\frac 12\hat R))\\
& = &\hat{p} _2 ^* \mathcal O (9\degree)\otimes \mbox {Nm} _2
(\holomorphic (\frac 12 (\hat R_2 -\hat R) )).
\end{eqnarray*}
The ramification divisors can naturally be expressed in terms of the
respective canonical divisors
\[
 \holomorphic (\hat R _2-\hat R) =K _{\hat\Sigma}\otimes \hat\pi _2 ^*K _ {\hat C _2} ^*\otimes (K _{\hat\Sigma}) ^*\otimes\zeta ^*K _ {\P ^1}.
\]
Then since $\mbox {Nm} _2\cdot \hat\pi _2 ^*=3I $ and $\zeta
=\hat{p} _2\cdot \hat\pi _2 $, we have
\begin{eqnarray*}
\mbox {Nm} _2 (\holomorphic (\frac 12 (\hat R_2 -\hat R))) & = & K_{\hat C_2}^{-3/2}\otimes \hat{p} _2 ^*K _ {\P ^ 1} ^ {3/2}\\
& = &K_{\hat C_2}^{-3/2}\otimes \hat{p} _2 ^*\holomorphic (-3).
\end{eqnarray*}
But the ramification divisor  $ S$ for the degree two map $\hat{p}
_2 $ satisfies $ S\simeq \hat{p} _ 2 ^*\holomorphic (6\degree) $, so
\[
K_{\hat C_2}^{-3/2}\simeq \hat{p} _2 ^* (\holomorphic (6\degree - 2)
^ {-3/2}\simeq \hat{p} _2 ^* (\holomorphic (-9\degree + 3).
\]
We have then shown that $\mbox {Nm} _2( \hat {\mathcal E}_ z
^*\otimes \holomorphic (-\frac 12\hat R)) =0 $, so that
\[
\hat {\mathcal E}_ z ^*\otimes \holomorphic (-\frac 12\hat R)\in
P(\hat\Sigma , \hat C _1)\cap P (\hat\Sigma, \hat C _2).
\]
\end {proof}

We shall now quotient by the six-fold symmetry $\tau $, which
clearly descends to the curves $\hat C _1 $ and $\hat C _2 $. Define
\[
C_j=\frac{\hat C _ j} {\tau},
\]
so that $ C _1 $ is given explicitly by $y ^ 3 - {b}_1 (\lambda) y ^
2+\frac {{b}_1 (\lambda) ^ 2} {4} y - {b}_2 (\lambda) = 0 $ and $ C
_2 $ by $z^2= {b} _ 2 (\lambda) $. We have the following commutative
diagram
\begin {equation}\label {commutative diagram}
\xymatrix{&\Sigma \ar[dl] _ {\pi_1}\ar[dr] ^ {\pi_2} &\\ C_1\ar[dr]
^ {{p} _1} & & C_2\ar[dl] _ {{p} _2}\\& P^1 &}
\end {equation}
with $\pi_1 (\eta,\lambda) = (\eta (\eta ^2 -\frac {{b} _ 1}
{2}),\lambda) $ and $\pi_2 (\eta,\lambda) = (\eta ^ 2,\lambda) $.

 Let $P( C _ 2,\P ^ 1) $, $ P(\Sigma,  C_1) $ be the Prym varieties of the 2-1 covers ${p} _2 $ and $\pi _1 $.
Notice that the norm map
\[
N_ {\pi _ 1}:[\sum n_i p_i ]\mapsto [\sum n_i\pi _1 (p_i) ]
\]
on equivalence classes of divisors is well defined on $ J
(\Sigma)/\pi _ 2 ^*J ( C_1) $. Thus we may define Tur$(\Sigma,\P ^
1) $ by the exact sequence
\[
0\rightarrow {\rm Tur}(\Sigma,\P ^ 1)\rightarrow P (\Sigma,  C _
1)\rightarrow P ( C _2,\P)\rightarrow 0.
\]
Alternatively Tur$(\Sigma,\P ^ 1) \subset P(\Sigma, C_2) $ consists
of those line bundles satisfying $\sigma ^*L\otimes L\simeq \mathcal
O $, where $\sigma (\eta,\lambda) = (-\eta,\lambda) $ is the
involution corresponding to the two-sheeted cover $\pi _
2:\Sigma\rightarrow C_2 $.  if $W $ is a subspace of  Jac ($\Sigma
$) we denote the set of elements of $ W $ satisfying $\sigma ^*
\mathcal  L\simeq \mathcal L $ by $ W ^ - $.  We will employ similar
notation for subspaces of Jac ($ C _ 2 $) which are anti-symmetric
with respect to the hyperelliptic involution. We have in particular
that
\[
{\rm Tur}(\Sigma,\P ^ 1) =P(\Sigma, C_2) ^ -=P(\Sigma, C_2) \cap
P(\Sigma , C _1).
\]
We denote by ${\rm Tur} _\R (\Sigma,\P ^ 1) $ the real subspace
given by the condition $\overline {\rho ^*{\ce}_z}\simeq\ce_z $.
Theorem~\ref {first Prym} and Theorem~\ref {second Prym} then
clearly yield a characterisation of the symmetry of the eigenline
bundles.  Griffiths's has given a criterium for when the eigenline
flow resulting from a Lax equation is linear \cite {Griffiths:85},
which Burstall showed was satisfied for harmonic maps into symmetric
spaces \cite {Burstall:92}.  Hence we have
\begin {corollary}
The  eigenline bundles l define a linear flow in a constant
translate of the intersection of the two Prym varieties defined by
\eqref {commutative diagram}:
\begin{eqnarray*}
\C &\rightarrow &  {\rm Tur} _\R (\Sigma,\P ^ 1)\\
z&\mapsto &{\mathcal E}_z\otimes \mathcal O(\frac 12  R)
\end{eqnarray*}
\end {corollary}
We now compute the dimension of this subvariety.
\begin {proposition} The intersection of the two Prym varieties Tur$(\Sigma,\P^1)$ in which the eigenline bundles flow has dimension $12k+3$.
\end {proposition}
\begin {proof} The discriminant of the polynomial defining $C _ 1\rightarrow\P^1$ is the same as that for the cover $\Sigma\rightarrow\P^1$ but in this case each zero of the discriminant yields only one ramification point and so by the Riemann-Hurwitz formula,
\[
2-2g_{C_1}=3\cdot 2-2(12k+2)-4
\]
and
\[
g _{C _ 1} = 12 k +2.
\]
Clearly,
\[
g_{C_2}=6 k.
\]
Thus from
\[
H ^ 1 (\Sigma,\mathcal O) ^ - =\pi _ {C _2} ^*H ^ 1 (C _ 2,\mathcal
O) ^ -\oplus T P (\Sigma, C _ 2) ^ -
\]
we calculate that
\begin{eqnarray*}
\mbox {dim (Tur}\Sigma,\P^1) =P (\Sigma, C _ 2) ^ -& = & (g _\Sigma - g _ {C_1}) - (g _ {C_2} -g _ {\P^1})\\
& = &12 k +3.
\end{eqnarray*}
\end {proof}

\bibliography {harmonic}
\bibliographystyle{alpha}
\end {document}